\documentclass[12pt]{amsart}

\usepackage{amsmath}
\usepackage{amsfonts}
\usepackage{amssymb}
\usepackage{graphicx, color}
\usepackage{nicefrac}

\newtheorem{theorem}{Theorem}[section]
\newtheorem{lemma}[theorem]{Lemma}
\newtheorem{corollary}[theorem]{Corollary}
\newtheorem{proposition}[theorem]{Proposition}
\theoremstyle{definition}

\theoremstyle{remark}

\def\zR{\mathbb R}
\def\zN{\mathbb N}
\def\M{\mathfrak{M}^1}

\begin{document}

\title[Energy integrals and metric embedding theory]{Energy integrals and metric embedding theory}

%\author{Daniel Carando\fnref{dani}}
%\ead{dgcarando@dm.uba.ar}
%\author{Daniel Galicer\corref{cor1}\fnref{dany}}
%\ead{dgalicer@dm.uba.ar}
%\fntext[dani]{Partially supported by ANPCyT PICT 05 17-33042, UBACyT Grant X038 and ANPCyT PICT 06 00587.}
%\fntext[dany]{Partially supported by ANPCyT PICT 05 17-33042, UBACyT Grant X863 and a CONICET doctoral fellowship.}
%
%\address{Departamento de Matem\'{a}tica, Facultad de Ciencias Exactas y Naturales, Universidad de Buenos Aires,\\ Pab. I, Cdad Universitaria (1428)  Buenos Aires, Argentina and CONICET.}
%

\author{Daniel Carando}
\author{Daniel Galicer}
\thanks{The first and second authors were partially supported by CONICET PIP 0624, PICT 2011-1456 and UBACyT 1-746. The third author was partially supported by  CONICET PIP 0624 and ANPCyT PICT 2011-0738}

\address{Departamento de Matem\'{a}tica, Facultad de Ciencias Exactas y Naturales, Universidad de Buenos Aires\\ Pab. I, Cdad Universitaria (C1428EGA)  Buenos Aires, Argentina and IMAS - CONICET.} \email{dcarando@dm.uba.ar}
\email{dgalicer@dm.uba.ar}

\author{Dami\'an Pinasco}

\address{Departamento de Matem\'{a}ticas y Estad\'{\i}sticas, Universidad T. Di Tella, Av. Figueroa Alcorta 7350 (C1428BCW), Buenos Aires, Argentina and CONICET}
\email{dpinasco@utdt.edu}



\maketitle

\begin{abstract}
For some centrally symmetric convex bodies $K\subset \mathbb R^n$, we study the energy integral
$$
\sup \int_{K} \int_{K} \|x - y\|_r^{p}\, d\mu(x) d\mu(y),
$$
where the supremum runs over all finite signed Borel measures $\mu$ on $K$ of total mass one. In the case where $K = B_q^n$, the unit ball of $\ell_q^n$ (for $1 < q \leq 2$) or an ellipsoid, we obtain the exact value or the correct asymptotical behavior of the supremum of these integrals.

We apply these results to a classical embedding problem in metric geometry. We consider in $\mathbb R^n$ the Euclidean distance $d_2$. For  $0 < \alpha < 1$, we estimate the minimum $R$ for which the snowflaked metric space $(K, d_2^{\alpha})$ may be isometrically embedded on the surface of a Hilbert sphere of radius~$R$.
\end{abstract}

\section{Introduction}

The study of integrals of the form
$$
\int_K\int_K f(x,y)\, d\mu(x)d\mu(y),
$$
where $\mu$ is a Borel measure supported on a compact set $K\subset \mathbb R^n$ was mainly motivated by problems in potential theory and geometric measure theory. Usually the function $f$ is given by  $f(x,y) = \| x - y \|_2^p$, where $p$ is a real number whose value (or range of values) depends on the problem under study. For example, if $p = 2 -n$ we have the classical Newtonian kernel.
Research on this type of integrals for other kernel functions (among them, the case where the exponent $p$ is positive) has its origin in the works of P\'olya and Szeg\"o \cite{PolSze31} and Schur \cite{Sch18} and have attracted many authors since then (see \cite{Ale77, Ale90, Ale91, AleSto74, Bjo56, Fej56, HinNicWol11, Kel70, NickolasWolf09, NickolasWolf11, NickolasWolf11b, Nielson65, Rogers94, Sperling60} and the references therein). Besides contributing to the classical potential theory, this study has allowed substantial progress in other areas such as discrepancy theory, metric inequalities and distance sums among others.
This work deals with this kind of integrals and its applications to metric geometry and embedding theory.

We denote by $\mathfrak{M}^1(K)$ the set of all finite signed Borel measures on the compact set $K\subset \mathbb R^n$ of total mass one. Given a continuous function $f:K \times K \to \zR$ and $\mu \in \mathfrak{M}^1(K)$, the \emph{$p-$energy integral of $K$} given by $\mu$ is defined by
$$
I_p(\mu,K,f) := \int_{K} \int_{K} f(x,y)^p \, d\mu(x) d\mu(y).
$$
The \emph{$p-$maximal energy of $K$} is given by
$$
M_p(K,f)=\sup \{I_p(\mu,K,f): \mu\in \mathfrak{M}^1(K)\}.
$$

In this note, we study energy integrals induced by $\ell_r$-norms, i.e., we consider the distance
functions $d_r(x, y) := \| x - y\|_r$. We focus our study on the case where $K$ is
a convex and balanced body. Several estimates are obtained for $K = B_q^n$, the unit ball of $\ell_q^n$, or in the case that $K$ is an ellipsoid in $\mathbb{R}^n$.

Alexander and Stolarsky \cite{AleSto74}, applying geometric arguments, computed the exact value of $M_1(B_2^1, d_2) = M_1([-1,1], d_2)$. A few years later, Alexander obtained the value  of $M_1(B_2^3, d_2)$ in \cite{Ale77} using  Archimedes' beautiful theorem on zonal areas (the usually called Hat-Box theorem). But it was Hinrichs, Nickolas and Wolf who took the big step: they managed to construct a sequence of measures on the ball $B_2^n$ whose marginals $w^*$-converge to the measure that maximizes the energy integral in the one dimensional case. They used this to calculate the precise value of $M_1(B_2^n, d_2)$ for every $n$.
Namely, they showed in \cite[Theorem 2.1]{HinNicWol11} that
\begin{equation}\label{HNW} M_1(B_2^n, d_2) =  \frac{\pi^{1/2}  \, \Gamma\left(\frac{n+1}{2}\right)  }{ \Gamma \left( \frac{n}{2}\right)}.\end{equation} We go further in this line, computing or estimating $M_p(K, d_r)$ for different subsets $K\subset \mathbb R^n$, different values of $p$ (as treated by Alexander and Stolarsky in \cite{AleSto74}) and different values of $r$.
For our purposes we use several tools and techniques from functional analysis and Banach space theory, such as the the theory of stable measures and $p$-summing operators.

Let us first introduce the following notation:
\begin{equation}\label{defidelmp}\mathfrak m_p:=M_p(B_2^1, d_2) = M_p([-1,1], d_2)= \sup \left\{I_p(\mu,[-1,1],d_2): \mu\in \mathfrak{M}^1([-1,1])\right\}.\end{equation}

For the Euclidean ball, we show the following formula (see Corollary~\ref{formula2}):
\begin{equation}\label{formula euclidean}
M_p(B_2^n, d_2) =  \mathfrak m_p  \, \frac{\pi^{1/2} \, \Gamma\left(\frac{n+p}{2}\right)  }{\Gamma\left( \frac{p+1}{2}\right) \Gamma \left( \frac{n}{2}\right)}.
\end{equation}

In fact, this will be a particular case of a result for ellipsoids in $\mathbb R^n$ given in Theorem~\ref{elipsoides}.  Note that, since $\mathfrak m_1=1$ (see \cite[Lemma 3.5]{AleSto74}), this formula recovers \eqref{HNW}.
Also, we show in Theorem~\ref{teo-norma-r} that $M_p(B_2^n, d_r)$  behaves asymptotically as $n^{\nicefrac{p}{r}}$ for $r\in [1,2)$ and $0<p<r$.

We consider convex and balanced bodies $K\subset \mathbb R^n$ as well. These sets can be seen as the unit ball of $\mathbb R^n$ with some norm. In Proposition~\ref{teorema promedio} we relate the value of $M_p(K, d_2)$  with a geometric property of $K$. Loosely speaking, let $W_t(K)$ be the width of $K$ in the direction $t$ (i.e., the distance between the supporting hyperplanes of $K$ orthogonal to $t$). Then, $M_p(K, d_2)$ can be controlled in terms of the average value of $W_t(K)^p$. If $K$ is the unit ball of an $n$-dimensional real Banach space $E = (\mathbb{R}^n, \| \, \|_E)$, we give in Proposition~\ref{cota inferior} a lower bound for the maximal $p$-energy $M_p(B_E, d_2)$ which is related with the $2$-summing norm of the identity operator from $E$ to $\ell_2^n$. This bound is obtained by calculating the supremum of the maximal energies over all the ellipsoids contained in $K$.
For $K=B_q^n$ and $1<q \leq 2$ we obtain, in Theorem~\ref{bolaeleq}, that $ M_p(B_q^n, d_2) $ behaves asymptotically as $n^{\nicefrac{p}{q'}}$, where $\frac{1}{q}+\frac{1}{q'}=1$.

\medskip

Let $K \subset \mathbb{R}^n$ be a compact set.
A classical result in metric geometry due to Schoenberg~\cite{Scho} asserts that, for $0< \alpha <1$, the snowflaked metric space $\left(K, d_2^\alpha\right)$ can be isometrically embedded in the surface of a Hilbert sphere (see Theorem~\ref{teorema geometrico} below for details).
Stated in another way, for every compact set $K \subset \mathbb{R}^n$, there exist a number $R$ and a mapping
\begin{equation}\label{isometria a la esferas}
j:(K, d_2^{\alpha}) \to (R \cdot S_{\ell_2}, \|\,\cdot\,\|_{\ell_2} )
\end{equation}
that preserves distances. Therefore, it is natural to ask for the least possible $R$ for which there exists a distance preserving mapping $j$ as in \eqref{isometria a la  esferas} (colloquially known as Schoenberg's radius for the metric space $\left(K, d_2^{\alpha}\right)$).
It was Alexander and Stolarsky \cite{AleSto74} who connected this problem, which is essentially a question in metric geometry, with the classical potential theory. They proved that the least possible radius can be computed from the $2\alpha$-maximal energy $M_{2\alpha}(K, d_2)$ (see also Theorem~\ref{teorema geometrico}).

We use this close connection and our results on energy integrals to compute or estimate these minimum radii for different compact sets $K$ in $\mathbb R^n$. As a consequence, we show in Theorem~\ref{crecimiento rho bolas} that, for $1 < q \leq 2$ and $0 < \alpha <1$, the minimum $R$ for which the metric space $(B_q^n, d_2^{\alpha})$  may be isometrically embedded on the surface of a Hilbert sphere of radius $R$ behaves asymptotically as  $n^{\nicefrac{\alpha}{q'}}$, where $\frac{1}{q} + \frac{1}{q'}=1$.

\section{Energy integrals: the results}\label{sec-euclidean}
\subsection{Energy integrals induced by the Euclidean distance}

In this section we prove formula \eqref{formula euclidean}. In fact, we will prove a more general result for ellipsoids, for which \eqref{formula euclidean} turns out to be a particular case.
An ellipsoid in $\mathbb R^n$ is the image of the Euclidean ball by a non-degenerate linear operator. The following theorem gives the value of $ M_p(\mathcal{E}, d_2)$ for $\mathcal{E}$ an ellipsoid. Its proof makes use of some properties of absolutely summing operators. We refer the reader to \cite{DefFlo93,DieJarTon95,Pie78,Tomczak-Jae} for definitions and a complete treatment of this subject.

It is important to mention that, for $p \geq 2$, we have $\mathfrak m_p = + \infty$. Indeed, if  $\mu_a$ is the measure supported on the points $-1, 0$ and $1$ with weights $a$, $1-2 a$ and $a$ respectively, we have $I_p(\mu_{a},[-1,1],d_2) \to +\infty$ as $a \to +\infty$. And, of course, also $M_p(K, d_2) = + \infty$ for every centrally symmetric convex body.
Therefore, we state all our results for the range $0 < p < 2$.

\begin{theorem}\label{elipsoides}
Let $T \in \mathcal{L}(\ell_2^n,\ell_2^n)$ be a bounded linear operator.
For $0 < p < 2$, we have
\[
M_p\left(T(B_2^n), d_2\right)=\mathfrak m_p\, \pi_p(T)^p,
\]
where $\pi_p(T)$ stands for the absolutely $p$-summing norm of the operator $T$.
\end{theorem}

For $I \in \mathcal{L}(\ell_2^n,\ell_2^n)$ the identity operator we have, by \cite[Theorem 11.10. and Exercise 11.24]{DefFlo93}, the equality
$$
\pi_p(I)^p=\frac{\pi^{1/2} \, \Gamma\left(\frac{n+p}{2}\right)  }{\Gamma\left( \frac{p+1}{2}\right) \Gamma \left(\frac{n}{2}\right)}.
$$
Using this and Stirling's formula \cite{Spira} we obtain the following result.

\begin{corollary}\label{formula2} Given $n \in \zN$, and $0 < p  <2$, then
\[
M_p(B_2^n, d_2)=\mathfrak m_p \, \frac{\pi^{1/2}\, \Gamma\left(\frac{n+p}{2}\right)}{\Gamma\left(\frac{p+1}{2}\right)\, \Gamma \left(\frac{n}{2}\right)} =   n^{\nicefrac{p}{2}} \left(\frac{\mathfrak m_p \, \pi^{1/2}}{\Gamma\left(\frac{p+1}{2}\right)} + o(1)\right).
\]
\end{corollary}

We remark that we can obtain a  direct and rather self-contained proof of Corollary~\ref{formula2}, which makes no use of absolutely summing operators. Indeed, we can proceed as in the proof of Theorem~\ref{elipsoides} and use \eqref{lemma piola} instead of Lemma~\ref{lemma piola2} below and obtain the corollary.

Some comments are in order. The absolutely $p$-summing norms for operators on Hilbert spaces are all equivalent to the Hilbert-Schmidt norm, with constants depending only on $p$ (as can be deduced from \cite[Corollay 3.16]{DieJarTon95}). Therefore, there exist positive constants $A_p,B_p$ such that for every $n$ and every ellipsoid $\mathcal{E}\subset\mathbb R^n$ with orthogonal axes and radii $a_1,\dots,a_n$, we have \begin{equation}\label{radioselipsoide}
A_p \left(a_1^2+\cdots+a_n^2\right)^{\nicefrac{p}{2}} \le M_p\left(\mathcal{E},d_2 \right) \le B_p \left(a_1^2+\cdots+a_n^2\right)^{\nicefrac{p}{2}}.
\end{equation}

We write $\lambda$ for the
normalized surface measure on the sphere $S^{n-1}$. This is an abuse of notation, since we have a different measure for each $n$, but there is no risk of confusion. Let $c_p^{(n)}$ be the $n$-dimensional $p$-th absolute moment defined as follows
$$
c_p^{(n)} = \left(\int_{S^{n-1}} |t_1|^p \, d\lambda(t) \right)^{1/p}.
$$
Although it  can easily be obtained by passing to spherical coordinates, we do not need the explicit value of $c_p^{(n)}$ for our purposes.
In order to prove the theorem, we need the following result, which extends \eqref{lemma piola} below and can be found in \cite[Lemma 10.5]{Tomczak-Jae}.
\begin{lemma} \label{lemma piola2}
Let $n$ be a positive integer, and $T \in \mathcal{L}(\ell_2^n,\ell_2^n)$. For $0 < p < \infty$ we have:
$$
\Vert Tx \Vert_2^p = (c_p^{(n)})^{-p} \int_{S^{n-1}} |\langle x, t \rangle|^p \, d\nu(t),
$$
where $\nu$ is the measure given by
$$
\int_{S^{n-1}} f(x) \, d\nu(x) = \int_{S^{n-1}} f\left(\frac{t}{\|T t\|_2}\right) \|T t\|_2^p \, d\lambda(t),
$$
for every continuous function on $S^{n-1}$.
\end{lemma}

Note that if $\nu$ is the measure defined in the previous lemma, by \cite[Proposition 10.4]{Tomczak-Jae} we have
\begin{equation}\label{ecunu}
\nu(S^{n-1}) = \int_{S^{n-1}} \|T t\|_2^p \, d\lambda(t) = (c_p^{(n)})^p \, \pi_p(T)^p.
\end{equation}

Two technical lemmas are also needed. First let us say that an atomic measure $\nu$ on $[-1,1]$ has  \emph{symmetric support} if it is supported in a set of points $-1 \leq p_1 < p_2 < \dots < p_{N-1} < p_N \leq 1$ with  $p_j = - p_{ N +1 - j}$ for every $j = 1 ,\dots, N$. If moreover $\nu([-1,1])=1$ and $\nu(p_j) = \nu( p_{ N +1 - j})$ for every $j = 1 ,\dots, N$ we say that $\nu$ is \emph{1-balanced}.
Such a measure can be written as  $\nu = \sum_{j=1}^N \lambda_j \delta_{p_j}$ on $[-1,1]$ with $\sum_{j=1}^N \lambda_j = 1$, $-1 \leq p_1 < p_2 < \dots < p_{N-1} < p_N \leq 1$ as above and $\lambda_j = \lambda_{N + 1 -j}$ for every $j = 1 ,\dots, N$.

The following lemma shows that 1-balanced measures are enough to compute $\mathfrak m_p$.

\begin{lemma} \label{alcanza con well distributed} For $\mathfrak m_p$ as in \eqref{defidelmp}, we have
$$
\mathfrak m_p =  \sup \int_{[-1,1]} \int_{[-1,1]} |u - v|^p \, d\nu(u) d\nu(v),
$$
where the supremum runs over all 1-balanced atomic measures $\nu \in [-1,1]$.
\end{lemma}

\begin{proof}
As in \cite[Lemma 3.3]{AleSto74}, it is easy to see that $\mathfrak m_p = \sup I_p(\nu, [-1,1], d_2)$, where the supremum runs over all atomic measures $\nu \in [-1,1]$ of total mass one. Moreover, the supremum can be taken within all atomic measures with symmetric support (adding points with weight zero if necessary).

Thus, to prove the lemma it is enough to show that, among all the measures of total mass one with support in a given symmetric set $\{p_1,\dots,p_N\}$,  $I_p(\cdot,[-1,1],d_2)$ attains its maximum at a 1-balanced one.
Define $\phi$ the quadratic form given by:
$$
\phi(x_1, \dots, x_N) = \sum_{i,j} x_i x_j d_2(p_i, p_j)^p = \sum_{i,j} x_i x_j \|p_i - p_j\|_2^p,
$$
and observe that, if $\nu$ is the measure given by  $\sum_{j=1}^N \lambda_j \delta_{p_j}$, we have $I_p(\nu, [-1,1], d_2)= \phi(\lambda_1, \dots, \lambda_N) = \sum_{i,j} \lambda_i \lambda_j \|p_i - p_j\|_2^p$.
By \cite[Theorem 3.3]{AleSto74} the quadratic form $\phi$ achieves a unique absolute maximum on the affine hyperplane $\sum_{j=1}^N x_j = 1$. Since the set $\{p_1,\dots,p_N\}$ is symmetric, it is easy to check that $\phi(\lambda_1, \dots, \lambda_N)=\phi(\lambda_N,\dots, \lambda_1)$ for all $\lambda$. This, together with the uniqueness of the maximum shows that the measure $\nu_0$ maximizing $I_p(\, \cdot \,,[-1,1],d_2)$ must be balanced.
\end{proof}

We do not have a closed formula for $\mathfrak m_p$. However, it is easy to obtain numerical estimations of this constant. Indeed, by the previous lemma we can compute $\mathfrak m_p$ as the supremum of some atomic measures. Thus, for each finite set of points $\{p_1,\dots,p_N\}\subset[-1,1]$, the maximum of the quadratic form $\phi$ in the lemma can be calculated by solving a linear system of equations (see \cite[Theorem 3.3]{AleSto74}).

\bigskip
We now define, for $t \in S^{n-1}$, the projection $\Pi_t : B_2^n \to [-1,1]$  by \begin{equation}\label{proyeccion}\Pi_t(x)=\langle x,t\rangle.\end{equation}
As in \cite{HinNicWol11}, we show how to relate our $n$-dimensional problem to the 1-dimensional one. We emphasize that the construction of the sequence of measures (see the statement below) draws heavily on the clever results of Hinrichs, Nickolas and Wolf.

\begin{lemma} \label{sucesion de medidas}
Let $e_1$ be the canonical unit vector $(1,0, \dots, 0) \in \mathbb{R}^{n}$ and $\nu$ be a 1-balanced atomic measure on $[-1,1]$. There exists a sequence of rotation invariant measures $(\eta_k)_{k \geq 1}$ in $\mathfrak{M}_1(B_2^{n})$ such that the sequence of projected measures $ \overline{\eta_k}:= \eta_k \Pi_{e_1}^{-1} \overset{w^*}{\longrightarrow} \nu$.
\end{lemma}

\begin{proof}
Fix any 1-balanced atomic measure $\nu = \sum_{j=1}^N \lambda_j \delta_{p_j}$ on $[-1,1]$.  By the proof of Lemmas 2.8 and 2.9 in \cite{HinNicWol11} (identifying the segment $[-1,1]$ with the diameter $D_{e_1}$ of the ball in the direction $e_1$) there exists a sequence of rotation invariant measures $(\mu_k)_{k}$ such that the sequence of projected measures $ \overline{\mu_k}  \overset{w^*}{\longrightarrow} \frac{1}{2} \left( \delta_{-1} +   \delta_{1} \right)$.
For $j=1, \dots, N$, set $\rho_j := |p_j|$. If $\rho_j > 0$ we define $\mu_k^j \in \mathfrak{M}_1(B_2^{n})$ the measure supported in $\rho_j B_2^{n}$ as:
$$
\mu_k^j(A) := \mu_k \left(\frac{1}{\rho_j} A\right),
$$
for every Borel set $A \in \rho_j B_2^{n}$. On the other hand, if $p_j=0$ we define $\mu_k^j = \delta_{p_j}=\delta_0$.
Observe now that, for every index $j=1, \dots, N$, the measure $\lambda_j \mu_k^j$ is rotation invariant and of total mass $\lambda_j$. Moreover, using that $p_j = - p_{N+1-j}$ (the support of $\nu$ is symmetric) and the fact that $\lambda_j = \lambda_{N+1-j}$ ($\nu$ is 1-balanced) we obtain that the sequence of projected measures $$
\lambda_j \overline{\mu_k^j} \overset{w^*}{\longrightarrow} \frac{1}{2} \left(\lambda_j \delta_{p_j} +  \lambda_j \delta_{p_{N+1-j}} \right) =  \frac{1}{2} \left(\lambda_j \delta_{p_j} +  \lambda_{N+1-j} \delta_{p_{N+1-j}} \right).
$$
If we set $\eta_k : = \sum_{j=1}^N \mu_k^j$ we have that $\eta_k$ is a rotation invariant measure  in $\mathfrak{M}_1(B_2^{n})$ and
$$
\overline{\eta_k} \overset{w^*}{\longrightarrow}  \sum_{j=1}^N \frac{1}{2} \left(\lambda_j \delta_{p_j} +  \lambda_{N+1-j} \delta_{p_{N+1-j}} \right) = \nu.  \qedhere
$$
\end{proof}

Now we are ready to prove our first theorem.

\begin{proof} \emph{(of Theorem~\ref{elipsoides})}

First observe that
$$
M_p\left(T(B_2^n), d_2\right) = \sup \int_{B_2^n} \int_{B_2^n} \Vert Tx - Ty\Vert_2^{p}\, d\mu(x) d\mu(y),
$$
where the supremum runs over all finite signed Borel measures $\mu$ on $B_2^n$ of total mass one.
Fix $\mu \in \mathfrak{M}_1(B_2^n)$ and set
$$
I_p(\mu;T):= \int_{B_2^n} \int_{B_2^n} \Vert Tx - Ty \Vert_2^p\, d\mu(x) d\mu(y).
$$
By Lemma~\ref{lemma piola2} we have
\begin{align*}
I_p(\mu;T)& = \int_{B_2^n} \int_{B_2^n}  (c_p^{(n)})^{-p} \int_{S^{n-1}} |\langle x - y, t \rangle |^p \, d\nu(t) d\mu(x) d\mu(y) \\
& = (c_p^{(n)})^{-p} \int_{S^{n-1}} \left[\int_{B_2^n} \int_{B_2^n} |\langle x - y, t \rangle |^p \, d\mu(x) d\mu(y) \right] d\nu(t).
\end{align*}
Now we use the notation introduced in \eqref{proyeccion} to get
\begin{align*}
I_p(\mu;T)&= (c_p^{(n)})^{-p}  \int_{S^{n-1}} \left[\int_{B_2^n} \int_{B_2^n} |\Pi_t(x) - \Pi_t(y)|^p \, d\mu(x) d\mu(y) \right] d\nu(t) \\
 & = (c_p^{(n)})^{-p}  \int_{S^{n-1}} \left[\int_{-1}^1 \int_{-1}^1 |u - v|^p \, d\mu\Pi_t^{-1}(u) d\mu\Pi_t^{-1}(v) \right] d\nu(t) \\
 & = (c_p^{(n)})^{-p} \int_{S^{n-1}} I_p(\mu\Pi_t^{-1}, [-1,1], d_2) \, d\nu(t).
\end{align*}
Note that $\mu\Pi_t^{-1}$ is also a finite signed Borel measure of total mass one on $[-1,1]$. Then,  we have $I_p(\mu\Pi_t^{-1},[-1,1], d_2) \leq \mathfrak m_p$ and
\begin{eqnarray*}
I_p(\mu;T) &\leq & (c_p^{(n)})^{-p} \int_{S^{n-1}} \mathfrak m_p \, d\nu(t) = \mathfrak m_p\, (c_p^{(n)})^{-p}  \, \nu(S^{n-1}) \\ & = & \mathfrak m_p \, \pi_p(T)^p,
\end{eqnarray*}
where the last equality follows by \eqref{ecunu}. This gives  $M\left(T(B_2^n), d_2\right) \leq  \mathfrak m_p \, \pi_p(T)^p$.

Let us show the reverse inequality. By standard manipulations it is easy to see that if $\mu$ is any rotation invariant measure then, for every $t \in S^{n-1}$, we have $I_p(\mu\Pi_t^{-1},[-1,1], d_2)= I_p(\mu\Pi_{e_1}^{-1}, [-1,1],d_2)$, where $e_1 = (1, 0, \dots, 0)$. Then,
\begin{eqnarray} \label{identidad}
I_p(\mu;T)& = & (c_p^{(n)})^{-p} \int_{S^{n-1}} I_p(\mu\Pi_t^{-1}, d_2)\, d\nu(t)= (c_p^{(n)})^{-p}\, I_p(\mu\Pi_{e_1}^{-1}, d_2)\, \nu(S^{n-1}) \\ &=& I_p(\mu\Pi_{e_1}^{-1},[-1,1], d_2)\, \pi_p(T)^p,\nonumber
\end{eqnarray}
for any rotation invariant measure $\mu$.
Given $\varepsilon > 0$, by Lemma~\ref{alcanza con well distributed} there is a 1-balanced measure $\zeta$ on $[-1,1]$ such that
$I_p(\zeta,[-1,1], d_2) \geq (1- \varepsilon) \mathfrak m_p$. With the help of Lemma~\ref{sucesion de medidas}, we can take a sequence $(\eta_k)_{k \geq 1}$ of rotation invariant measures on $B_2^n$ such that $\eta_k \Pi_{e_1}^{-1} \overset{w^*}{\longrightarrow} \zeta$.
As in \cite[Corollary 2.7]{NickolasWolf09}, it is easy to see that $I_p(\,\cdot\,,[-1,1], d_2)$ is $w^*$-sequentially continuous on the set of all Borel measures on $D_{e_1}$. Hence,
$$
M_p\left(T(B_2^n), d_2\right) \geq I_p(\eta_k;T) =  I_p(\eta_k\Pi_{e_1}^{-1}, [-1,1], d_2)\, \pi_p(T)^p
$$
and the latter tends to $I_p(\zeta,[-1,1], d_2) \, \pi_p(T)^p \ge (1- \varepsilon) \mathfrak m_p \, \pi_p(T)^p.$
Since $\varepsilon$ is arbitrary, we have shown that $M_p\left(T(B_2^n), d_2\right) = \mathfrak m_p \, \pi_p(T)^p$.
\end{proof}

\subsection{Energy integrals induced by the $\ell_r-$norm $(1 \le r <2)$} \label{norma-r}

Now, we deal with the estimates of the $p-$maximal energy of the Euclidean ball induced by the distance functions $d_r(x,y)=\Vert x-y\Vert_r$ for $r\in [1,2)$.
We do not analyze the case where $r >2$ since in this case, $(\mathbb{R}^n, d_r)$ is not a quasihypermetric space (see \cite{NickolasWolf09} and the references therein) and then the corresponding energy integrals are not uniformly bounded.
Our goal is to prove the following asymptotic behavior of $M_p(B_2^n, d_r)$. Note that for $r=2$ this is contained in Corollary~\ref{formula2}.

\begin{theorem}\label{teo-norma-r} Let $ r\in[1,2]$ and $p \in (0,r)$, then $M_p(B_2^n, d_r)$ behaves asymptotically as~$n^{\nicefrac{p}{r}}.$
\end{theorem}%

Now we recall the basic property of stable measures (see, for example, \cite[Lemma 21.1.3]{Pie78},  \cite[Theorem 6.4.15 through Theorem 6.4.18]{AlbKal06} and  \cite[Section 24]{DefFlo93}).  For any $n\in \zN$ and $r\in [1,2)$, there exists a measure $m^n_r$ (called the $r-$stable measure) defined on the Borel sets of $\zR^n$ such that
\begin{equation}\label{lemma piolar}
\Vert x \Vert_r^p=c_{r,p}^{-p} \int_{\zR^n} \vert \langle x,w \rangle\vert^p \, dm^n_r(w)
\end{equation}
for all $p\in (0,r)$. Here $c_{r,p}$ is the $p$th moment of the one dimensional stable measure $m^1_r$, namely, $$
c_{r,p}= \left(\displaystyle{\int_{\zR} \vert w \vert^p \, dm^1_r(w)}\right)^{1/p}.
$$
The exact value of $c_{r,p}$ can be found, for example, in \cite[21.1.2]{Pie78}.

For $r=2$, the measure $m^n_r$ is just the $n$-dimensional Gaussian measure $\gamma_n$.
For this measure, we actually have for $x \in \mathbb{R}^n$ and \emph{every} $0 < p < \infty$,
\begin{equation}\label{lemma piola}
\Vert x \Vert_2^p = c_{2,p}^{-p} \int_{\zR^n} |\langle x,  t \rangle|^p \, d\gamma_n(t)= b_p^{(n)} \int_{S^{n-1}} |\langle x,  t \rangle|^p \, d\lambda(t)
\end{equation}
and
$$
c_{2,p}=\left(\displaystyle{\int_{\mathbb R}} |w|^p \, d\gamma_1(w) \right)^{1/p} = 2 \left(\frac{\Gamma(\frac{1+p}{2})}{\Gamma(\frac{1}{2})} \right)^{1/p}\quad\text{;}\quad
b_p^{(n)} = \frac{\pi^{1/2} \,\Gamma\left(\frac{n+p}{2}\right)  }{\Gamma\left( \frac{p+1}{2}\right) \Gamma \left( \frac{n}{2}\right)}.
$$
The first equality in \eqref{lemma piola} is just the stability property of Gaussian measures. The second equality follows using spherical coordinates. Note that the equality between the first and the third expressions is just  Lemma~\ref{lemma piola2} applied to the identity operator in $\mathbb R^n$.

Using $r-$stable measures it is possible to obtain upper bounds for $M_p(B_E, d_r)$. Note that, in the one dimensional case, we have $d_r = d_2$. Therefore, the energies induced on $[-1,1]$ by all these distance functions  obviously agree.

\begin{lemma}\label{GUB} Let $E$ be an $n-$dimensional Banach space. If $r\in [1,2]$ and $0<p<r$, then
\[
M_p(B_E, d_r) \leq  \mathfrak m_p\, c_{r,p}^{-p}\,  \int_{\mathbb \zR^{n}} \|t\|_{E^{\prime}}^p \, dm^n_r(t),
\]
where $m^n_r$ is the $n-$dimensional $r-$stable measure.
\end{lemma}

\begin{proof} Let us apply  \eqref{lemma piolar} to compute the $p-$maximal energy of $B_E$ induced by $d_r$.
\begin{align*}
M_p(B_E, d_r)= & \sup_{\mu \in \M(B_E)}\ \int_{B_E} \int_{B_E} \|x -y \|_{r}^p \, d\mu(x)d\mu(y) \\
= & \sup_{\mu \in \M(B_E)}  \int_{B_{E}} \int_{B_{E}} \left[c_{r,p}^{-p} \int_{\zR^n}  |\langle x - y, t \rangle |^{p} \, dm^n_r(t)\right] \, d\mu(x) d\mu(y) \\
= & \sup_{\mu \in \M(B_E)}  c_{r,p}^{-p} \int_{\zR^n} \left[\int_{B_{E}} \int_{B_{E}} |\langle x - y, t \rangle |^{p} \, d\mu(x) d\mu(y) \right] dm^n_r(t) \\
= & \sup_{\mu \in \M(B_E)} c_{r,p}^{-p} \int_{\zR^n} \Vert t \Vert_{E^{\prime}}^p \left[\int_{B_{E}} \int_{B_{E}} \left|\left\langle x - y, \dfrac{t}{\Vert t \Vert_{E^{\prime}}} \right\rangle\right|^{p}\,  d\mu(x) d\mu(y) \right] dm^n_r(t) \\
= & \sup_{\mu \in \M(B_E)}  c_{r,p}^{-p} \int_{\zR^n} \Vert t \Vert_{E^{\prime}}^p \left[\int_{-1}^1 \int_{-1}^1 |u-v|^{p} \, d\mu_{t/\Vert t \Vert_{E^{\prime}}}(u) d\mu_{t/\Vert t \Vert_{E^{\prime}}}(v) \right] dm^n_r(t) \\
\le &  \mathfrak m_p\, c_{r,p}^{-p}\,  \int_{\mathbb \zR^{n}} \|t\|_{E^{\prime}}^p \, dm^n_r(t). \qedhere
\end{align*}
\end{proof}

We can also use \eqref{lemma piola} to derive an upper bound of $M_p(B_E, d_2)$ using the average over the unit sphere. We will return later to this point because it is possible to obtain some geometric properties of $K$ related with its average width.

We set some useful notation first. Let $(a_n)_{n \in \mathbb{N}}$ and $(b_n)_{n \in \mathbb{N}}$ be sequences of non-negative numbers. If there are positive constants $A$ and $B$ such that $A b_n \leq a_n \leq B b_n$ for every $n$, we write $ (a_n)_{n \in \mathbb{N}} \asymp (b_n)_{n \in \mathbb{N}}$ . On the other hand, we write $ (a_n)_{n \in \mathbb{N}} \preccurlyeq (b_n)_{n \in \mathbb{N}}$ or $ (b_n)_{n \in \mathbb{N}} \succcurlyeq (a_n)_{n \in \mathbb{N}}$, in the case there is a positive constant $C$ such that $a_n \leq C b_n$ for every natural number $n$.

In the following lemma we present the asymptotic behavior of the average of powers of $\ell_r-$norms on the unit sphere.

\begin{lemma} \label{lema promedio equivalente}
Given $n\in \zN, \ p>0 $ and $1 \leq r < \infty$, we have
\[
\int_{S^{n-1}} \| t \|_r^p \, d\lambda(t) \asymp n^{(\frac{1}{r} - \frac{1}{2})p}.
\]
\end{lemma}

\begin{proof}
We define $\varphi^n_r: \zR_{+} \to \zR$ by
\[
\varphi^n_r(p)=\int_{S^{n-1}} \left[n^{( \frac{1}{2}-\frac{1}{r})}\, \| t \|_r\right]^p \, d\lambda(t).
\]
Then, what we have to prove is that $\varphi^n_r(p) \asymp 1$.

Standard computations show that $$\int_{\mathbb{R}^n}  \| x \|_r^r \, d\gamma_n(x) = \sum_{i=1}^n  \int_{\mathbb{R}^n} |x_i|^r\, d\gamma_n(x) = n\, c_{2,r}^r.$$ Therefore, using spherical coordinates and Stirling's formula we obtain
$$
\int_{S^{n-1}} \| t \|^r_r \, d\lambda(t) \asymp n^{1-\frac{r}{2}}
$$
or, equivalently,
$$
\varphi^n_r(r) = \int_{S^{n-1}} \left[n^{\frac{1}{2}-\frac{1}{r}} \, \| t \|_r\right]^r \, d\lambda(t)   \asymp 1.
$$
This gives the desired result for the particular case $p=r$.

We now consider $2 \leq r < \infty$. In this case, for every $t \in S^{n-1}$ we have
$$
1=\| t \|_2 \le n^{(\frac{1}{2}-\frac{1}{r})} \| t\|_r,
$$ which gives the lower bound
$
 \varphi^n_r(p) \ge 1
$ for every $p>0$
and shows that  $\varphi^n_r$ is an increasing function of $p$. As a consequence, for $0<p \le r$ we have $ 1\le  \varphi(p) \le \varphi(r) \preccurlyeq 1$.

For $p> r$, note that
\[
\| t \|_r \le n^{(\frac{1}{r}-\frac{1}{p})} \|t\|_p  \]
and then $$\left[n^{( \frac{1}{2}-\frac{1}{r})}\, \| t \|_r\right]^p \le \left[n^{( \frac{1}{2}-\frac{1}{r})}\, n^{(\frac{1}{r}-\frac{1}{p})} \|t\|_p\right]^p=\left[n^{( \frac{1}{2}-\frac{1}{p})} \|t\|_p\right]^p.
$$
Therefore,
\[
1\le \varphi^n_r(p) = \int_{S^{n-1}} \left[n^{( \frac{1}{2}-\frac{1}{r})}\, \| t \|_r\right]^p \, d\lambda(t) \le \int_{S^{n-1}} \left[n^{( \frac{1}{2}-\frac{1}{p})}\, \| t \|_{\ell_p}\right]^p \, d\lambda(t) \preccurlyeq 1,
\]
which concludes the proof for $2 \leq r < \infty$.

Suppose now that $1 \leq r <2$. Then, for every $t \in S^{n-1}$ we have
\begin{equation}\label{relacion2}
n^{\nicefrac{1}{2} - \nicefrac{1}{r}} \|t\|_{r} \leq 1
\end{equation}
and thus
$\varphi^n_r(p) \leq 1$ for every $p>0$. For the reverse inequality, consider first $0<p<1$. Since $\varphi^n_r $ is a decreasing function and $r\ge 1$, we get  $1  \asymp \varphi^n_r(r) \leq \varphi^n_r(p)$ as above.
For  $p\ge 1$, using spherical coordinates and H\"{o}lder inequalities we obtain
\begin{align*}
\left(\int_{S^{n-1}} \| t \|_r^p \, d\lambda(t)\right)^{1/p} & \geq   \int_{S^{n-1}} \| t \|_r\, d\lambda(t)  \asymp \frac{1}{\sqrt{n}}  \int_{\mathbb{R}^n} \| x \|_r\, d\gamma_n(x) \\
&  \geq \frac{1}{\sqrt{n}} \left( \sum_{i=1}^n \left| \int_{\mathbb{R}^n} |x_i|\, d\gamma_n(x) \right|^r  \right)^{1/r}  \asymp n^{1/r - 1/2}.
\end{align*}
This shows that $\varphi^n_r(p) \succcurlyeq 1$ and ends the proof.
\end{proof}

Now we can prove the main result of this section.
\begin{proof} \emph{(of Theorem~\ref{teo-norma-r})}

We can apply Lemma~\ref{GUB} in the particular case where $E$ is the $n-$dimensional Euclidean space. Then,
\[
M_p(B_2^n, d_r) \leq \mathfrak m_p\, c_{r,p}^{-p}\,  \int_{\mathbb \zR^{n}} \|t\|_{2}^p \, dm^n_r(t).
\]
We can estimate this last integral writing the $\ell_2-$norm as an average on the unit sphere, so
\begin{align*}
M_p(B_2^n, d_r) \leq & \, \mathfrak m_p\, c_{r,p}^{-p}\  \int_{\mathbb \zR^{n}} \|t\|_{2}^p \, dm^n_r(t) \\
= & \,\mathfrak m_p\, c_{r,p}^{-p}\   \int_{\mathbb \zR^{n}} b_p^{(n)} \int_{S^{n-1}} |\langle x,  t \rangle|^p d\lambda(t) \, dm^n_r(t) \\
= & \,\mathfrak m_p\, b_p^{(n)}\   \int_{S^{n-1}} c_{r,p}^{-p} \int_{\mathbb \zR^{n}} |\langle x,  t \rangle|^p \, dm^n_r(t)\, d\lambda(t) \\
= & \, \mathfrak m_p\, b_p^{(n)}\  \int_{S^{n-1}} \|w\|_r^p \, d\lambda(t) .
\end{align*}
Recall that $b_p^{(n)} = \frac{\pi^{1/2} \,\Gamma\left(\frac{n+p}{2}\right)  }{\Gamma\left( \frac{p+1}{2}\right) \Gamma \left( \frac{n}{2}\right)} \asymp n^{\nicefrac{p}{2}}.$ Thus, by Lemma~\ref{lema promedio equivalente}, we have
\[
M_p(B_2^n, d_r) \preccurlyeq n^{\nicefrac{p}{2}} n^{(\nicefrac{1}{r} - \nicefrac{1}{2})p}=n^{\nicefrac{p}{r}}.
\]
In order to prove the reverse inequality, as in the proof of Theorem~\ref{formula2} and using Lemma~\ref{sucesion de medidas}, given $\varepsilon >0$ we can find a sequence of rotation invariant measures $(\eta_k)_{k}$ such that
\begin{align*}
M_p(B_2^n, d_r)= & \sup_{\mu \in \M(B_2^n)} c_{r,p}^{-p} \int_{\zR^n} \Vert t \Vert_{2}^p \left[\int_{-1}^1 \int_{-1}^1 |u-v|^{p} \, d\mu_{t/\Vert t \Vert_{2}}(u) d\mu_{t/\Vert t \Vert_{2}}(v) \right] dm^n_r(t) \\
\ge & \limsup_{k \to \infty} c_{r,p}^{-p} \int_{\zR^n} \Vert t \Vert_{2}^p \left[\int_{-1}^1 \int_{-1}^1 |u-v|^{p} \, d(\eta_k)_{t/\Vert t \Vert_{2}}(u) d(\eta_k)_{t/\Vert t \Vert_{2}}(v) \right] dm^n_r(t) \\
\ge & \mathfrak m_p\,(1- \varepsilon) \, c_{r,p}^{-p}\, \int_{\zR^n} \Vert t \Vert_{2}^p \, dm^n_r(t) \asymp n^{\nicefrac{p}{r}}. \qedhere
\end{align*}
\end{proof}

\subsection{Energy integrals on the ball of  ${\ell_q^n}$ }\label{sec-convex}

Given a centrally symmetric convex body $K$, we have some general upper estimates for the energies induced by the different $\ell_r-$norms as in Lemma~\ref{GUB}. However it seems to be difficult compute the exact value of $M_p(K, d_r)$ or its asymptotic behavior. In this section we deal with $\ell_q$-balls for $1<q<2$.

{A combination of Lemma~\ref{GUB} and \eqref{lemma piola} gives the following result, which has some geometrical interest.}

\begin{proposition} \label{teorema promedio}
Let $E = (\mathbb{R}^n, \| \, \|_E)$ be a real Banach space of dimension $n$. For $0 < p < 2$, we have
$$
M_p(B_E, d_2) \leq \mathfrak m_p \, b_p^{(n)} \,  \int_{S^{n-1}} \| t\|_{E'}^p \, d\lambda(t),
$$
where $E'$ denotes the dual of $E$.
\end{proposition}

This proposition has the following geometrical interpretation. Let $W_t=W_t(B_E)$ be the width of $B_E$ in the direction $t$. This is defined to be  the (Euclidean) distance between the supporting hyperplanes of $B_E$ orthogonal to $t$, and can be computed as
$$
W_t= \sup_{x\in B_E} \langle x,t\rangle -  \sup_{x\in B_E} \langle x, - t\rangle = 2\|t\|_{E'}.
$$
As a consequence, we have established a relationship between $ M_p(B_E, d_2)$ and the expected value of $W_t^p$, which is a kind of average width of $B_E$.

Let $E = (\mathbb{R}^n, \| \, \|_E)$ be a real Banach space of dimension $n$, we now give a lower bound for the energy integral $M_p(B_E, d_2)$ which is related with the $2$-summing norm of the identity operator from $E$ to $\ell_2^n$. This bound is obtained by computing the supremum of the maximal energies over all the ellipsoids contained in $B_E$.

\begin{proposition} \label{cota inferior}
There exist a positive constant $C_p$ such that for every real $n$-dimensional Banach space  $E = (\mathbb{R}^n, \| \, \|_E)$ we have
$$
C_p \,\, \pi_2 \left( i_{E,2}: E \to \ell_2^n \right)^p \leq M_p(B_E, d_2),
$$
where $i_{E,2}: E \to \ell_2^n$ is the formal identity.
\end{proposition}

\begin{proof}
Observe that, if $S : \ell_2^n \to E$ is an operator of norm one, then the ellipsoid $S(B_{2}^n)$ is contained in $B_E$. Therefore, $ M_p(S(B_2^n), d_2) \leq M_p(B_E, d_2).$ Using Theorem~\ref{elipsoides} we obtain that $\mathfrak{m}_p \, \pi_p(i_{E,2}S : \ell_2^n \to \ell_2^n )^p = M_p(S(B_2^n), d_2) \leq M_p(B_E, d_2).$ Since the absolutely $p$-summing norms for operators on Hilbert spaces are all equivalent to the $2$-summing norm (as can be deduced from \cite[Corollay 3.16 and Theorem 4.10]{DieJarTon95}) we know there exists a constant $A_p$ which depends only on $p$ such that
$A_p \pi_2(i_{E,2}S : \ell_2^n \to \ell_2^n ) \leq \pi_p(i_{E,2}S : \ell_2^n \to \ell_2^n ).$
Thus,
\begin{equation}\label{eqsummingenergia}
\underbrace{A_p^p \, \mathfrak{m}_p}_{C_p} \, \pi_2(i_{E,2}S : \ell_2^n \to \ell_2^n )^p \leq M_p(B_E, d_2).
\end{equation}
Since equation \eqref{eqsummingenergia} holds for every norm one operator $S : \ell_2^n \to E$ we obtain
\begin{equation}\label{eqsummingenergia2}
C_p \, \sup \left\{\pi_2(i_{E,2} S : \ell_2^n \to \ell_2^n)^p\, : S \in \mathcal{L}(\ell_2^n, E), \ \|S\| =  1\right\}  \leq M_p(B_E, d_2).
\end{equation}

Now by Kwapie\'n's test \cite[Proposition 11.8]{DefFlo93} we have that
$$
\pi_2(i_{E,2} : E \to \ell_2^n)= \sup\left\{ \pi_2(i_{E,2} S : \ell_2^n \to \ell_2^n) \,:  S \in \mathcal{L}(\ell_2^n, E), \ \|S\| = 1\right\}.
$$
Therefore, by \eqref{eqsummingenergia2}, we get
$$
C_p \, \pi_2 \left(i_{E,2}: E \to \ell_2^n \right)^p \leq M_p(B_E, d_2).
$$
This concludes the proof.
\end{proof}

We now describe the asymptotical behavior of $ M_p(B_q^n, d_2) $ for $1 < q\le 2$.

\begin{theorem}\label{bolaeleq}
Given $1 < q\le 2$ and $p>0$, then $ M_p(B_q^n, d_2) $ behaves asymptotically as $n^{p/q'}$, where $\frac{1}{q}+\frac{1}{q'}=1$.
\end{theorem}
\begin{proof} By Proposition~\ref{teorema promedio} and Lemma~\ref{lema promedio equivalente}, we have
\[
M_p(B_q^n, d_2) \leq  \mathfrak m_p \, b_p^{(n)} \,  \int_{S^{n-1}} \| t\|_{q'}^p \, d\lambda(t) \asymp n^{\nicefrac{p}{q^{\prime}}}.
\]

On the other hand, by Proposition~\ref{cota inferior}, we know that $\pi_2 \left(\ell_q^n \to \ell_2^n \right)^p \preccurlyeq M_p(B_E, d_2).$
Now, by \cite[Lemma 22.4.9]{Pie78} or \cite[Theorem 1]{BauLin77}, we know that $\pi_2 \left( \ell_q^n \to \ell_2^n \right)^p \asymp n^{\nicefrac{p}{q^{\prime}}}$ which gives the lower estimate.
\end{proof}

Some comments are in order. It should be mentioned that we can avoid the use of Proposition~\ref{cota inferior} for the lower estimate in the previous theorem. Indeed, since $n^{(q-2)/2q}\, B_2^n \subset  B_q^n$, we have
\[
n^{\nicefrac{p}{q^{\prime}}} =  n^{p(q-2)/2q} n^{\nicefrac{p}{2}} \asymp n^{p(q-2)/2q} M_p(B_2^n, d_2) = M_p(n^{(q-2)/2q}\, B_2^n, d_2)  \le  M_p(B_q^n, d_2).
\]

Since we have obtained the correct asymptotic estimate of $ M_p(B_q^n, d_2) $ by using Propositions~\ref{teorema promedio} and \ref{cota inferior}, this says that the bounds of their statement cannot be improved for arbitrary spaces.

Although we have not obtained the asymptotic behavior  of $ M_p(B_q^n, d_2) $ for the remaining values of $q$, we do have certain bounds. Note that, for $q >2$, we have the inclusions $B_2^n \subset B_q^n \subset n^{\nicefrac{1}{2} - \nicefrac{1}{q}} B_2^n$. Therefore, $n^{\nicefrac{p}{2}} \preccurlyeq M_p(B_q^n, d_2) \preccurlyeq n^{\nicefrac{p}{q^{\prime}}}$. It is interesting to mention that this bounds are the same as the ones that can be obtained using Proposition~\ref{cota inferior} and Proposition~\ref{teorema promedio}.

From Theorems \ref{bolaeleq} and \ref{elipsoides} and Proposition~\ref{teorema promedio}, we can find a relationship between $M_p(K)$ and  the expected value of a random width of $K$, for $K$ the unit ball of $\ell_q^n$ or an ellipsoid in $\mathbb R^n$. Namely, suppose $t\in S^{n-1}$ is randomly chosen with uniform distribution in the sphere. Then, for  $0 < p < 2$ and $1< q\le 2$ we have
\[
M_p(B_{q}^n, d_2) \asymp n^{p/2}\, \mathbb E(W_t(B_{q}^n)^p).
\]

A similar result holds for ellipsoids. Moreover, in this case, the cited results and \eqref{radioselipsoide} give us constants $\tilde A_p, \tilde B_p>0$ with the following property. For every $n \in \mathbb{N}$ and every ellipsoid $\mathcal{E}\subset \mathbb R^n$ with orthogonal axes and radii $a_1,\dots,a_n$ we have
$$
\tilde A_p \left(\frac{a_1^2+\cdots+a_n^2}{n}\right)^{\nicefrac{1}{2}} \le \mathbb E(W_t(\mathcal{E})^p)^{\nicefrac{1}{p}} \le \tilde B_p \left(\frac{a_1^2+\cdots+a_n^2}{n}\right)^{\nicefrac{1}{2}}.
$$

\subsection{Metric embeddings}\label{embeddings}

Uniform, Lipschitz and coarse embeddings of metric
spaces into Banach spaces with ``good geometrical properties'' have found many significant applications,
specially in computer science and topology. The advantages of these embeddings are based on the fact that for spaces with ``good properties'' one can
apply several geometric tools which are generally not available for typical metric
spaces. The most significant accomplishments throughout these
lines were obtained in the area of approximation algorithms. In this
context, the spaces with ``good geometrical properties''
are mostly a separable Hilbert space (or certain classical Banach spaces, such as $\ell_1$).

One of the fundamental problems in metric geometry is the immersion problem, i.e., to determine conditions for which a metric space may be isometrically embedded in a Hilbert space.
It is well known that not every metric space $(X,d)$ can be isometrically embedded in a Hilbert space. Even if distortion is allowed, there are metric spaces that cannot be embedded in a Hilbert space. The celebrated Assouad's embedding theorem~\cite{Ass83} allows a bi-Lipschitz embedding if we change  the metric a little bit. Let $(X,d)$ be a  \emph{doubling} metric space and take $\alpha\in(0,1)$. Assouad's theorem states that there exists $N$ such that the snowflaked metric space $(X,d^\alpha)$ admits a bi-Lipschitz embedding into $\zR^N$ endowed with the Euclidean norm (i.e, in a $N$-dimensional Hilbert space). Recently, Naor and Neiman  \cite{NaoNei12} proved that the same dimension $N=N(k)$ can be chosen for all $\alpha>\nicefrac 1 2$ and all metric spaces with doubling constant at most $k$. Moreover, the distortion of all the corresponding bi-Lipschitz embeddings is uniformly bounded.

We will concentrate our attention to isometric embeddings for certain snowflaked metric spaces.

In the early twentieth century, Wilson \cite{Wilson} investigated those metric spaces which arise from a metric space by taking as its new metric a suitable (one variable) function of the old one.
For the metric space $(\mathbb{R},d_2)$, he considered the metric transform $f(t)=t^{1/2}$ and showed that the snowflaked metric space $(\mathbb{R},d_2^{1/2})$ can be isometrically embedded  in a separable Hilbert space.
In other words, he showed the existence of a distance preserving mapping $j:\left(\mathbb{R},d_2^{1/2}\right) \to \left(\ell_2, \| \,\cdot\, \|_{\ell_2}\right).$
Some years later, von-Neumann and Schoenberg \cite{SchoVN} characterized those functions $ f $ for which the metric space $\left(\mathbb{R}, f(d_2)\right)$ can be isometrically embedded in a Hilbert space. As a particular case they proved that, for $\alpha \in (0,1)$, the function $f(t)=t^{\alpha}$   is an appropriate metric transform, generalizing Wilson's result.
Using transcendental means, Schoenberg obtained in \cite{Scho} the same result for $\mathbb{R}^n$. He proved that, for $0 < \alpha <1$, the metric space $\left(\mathbb{R}^n, d_2^\alpha\right)$ can  also be embedded in $\ell_2$.

In particular, for every compact set $K \subset \mathbb{R}^n$ and every $0< \alpha <1$, the snowflaked metric space $\left(K, d_2^\alpha\right)$ can be isometrically embedded in the surface of a Hilbert sphere (see Theorem~\ref{teorema geometrico} below for details).
In other words, there exist a number $R$ and a distance preserving mapping \begin{equation}\label{isometria a la esfera}
j:(K, d_2^{\alpha}) \to (R \cdot S_{\ell_2}, \|\,\cdot\,\|_{\ell_2} ).
\end{equation}
We now focus on finding (or estimating) the Schoenberg radius of $\left(K, d_2^{\alpha}\right)$, i.e.,  the least possible $R$ for which there exists a distance preserving mapping $j$ as in \eqref{isometria a la  esfera}. Alexander and Stolarsky \cite{AleSto74} connected the problem of estimating this radius with the calculation of certain energy integrals. We state this relation in the following theorem.

\begin{theorem}\label{teorema geometrico}
Let $K \subset \mathbb{R}^n$ be a convex body. For $0 < \alpha < 1$, the snowflaked metric space $(K, d_2^{\alpha})$ may be isometrically embedded on the surface of a Hilbert sphere of radius $ \sqrt{\frac{M_{2\alpha}(K, d_2)}{2}}.$ Moreover, this is the minimum possible radius.
\end{theorem}

Actually, Alexander and Stolarsky proved this result for  $\alpha=\nicefrac 1 2$, but their proof works almost line by line for $0 < \alpha < 1$. Similar results can also be found in \cite[Theorems 3.1., 3.2. and 4.6]{NickolasWolf11b}.

As a consequence of the last theorem and Corollary~\ref{formula2}, we obtain the minimum $R$ for which the metric space $(B_2^n, d_2^{\alpha})$ may be isometrically embedded on the surface of a Hilbert sphere of radius $R$.
\begin{theorem}\label{crecimiento rho bola euclidea}
For $0 < \alpha < 1$, the minimum $R$ for which the metric space $(B_2^n, d_2^{\alpha})$ may be isometrically embedded on the surface of a Hilbert sphere of radius $R$ is
$$
\sqrt{\frac{\mathfrak m_p \, \pi^{1/2}\, \Gamma\left(\alpha+\frac{n}{2} \right)}{2\, \Gamma\left(\alpha + \frac{1}{2}\right) \Gamma \left(\frac{n}{2}\right)}} = n^{\alpha/2} \left(\sqrt{\frac{ \mathfrak m_p \, \pi^{1/2}}{2\, \Gamma\left(\alpha + \frac{1}{2}\right)}}+o(1)\right).
$$
\end{theorem}

On the other hand, combining Theorem~\ref{bolaeleq} with Theorem~\ref{teorema geometrico} we can obtain the asymptotical behavior of the Schoenberg radius for the metric space $(B_q^n, d_2^{\alpha})$.

\begin{theorem}\label{crecimiento rho bolas}
Let $1 < q \leq 2$ and $0 < \alpha <1$, the minimum $R$ for which the metric space $(B_q^n, d_2^{\alpha})$  may be isometrically embedded on the surface of a Hilbert sphere of radius $R$ behaves asymptotically as  $n^{\nicefrac{\alpha}{q'}}$, where $\frac{1}{q} + \frac{1}{q'}=1$.
\end{theorem}

%
%
%\subsection*{Acknowledgements} We would like to thank the anonymous referee for his/her useful suggestions and comments, and for pointing out an error in a proof in a previous version of the manuscript.

\end{document}